\documentclass[reqno,12pt,a4paper]{amsart}
\usepackage{bbm}
\usepackage{}

\usepackage{tikz}
\usetikzlibrary{matrix,arrows,calc,snakes,patterns,decorations.markings}

\usepackage{booktabs}
\usepackage[mathscr]{eucal}
\usepackage{graphics,epic}
\usepackage{amsfonts}
\usepackage{amscd}
\usepackage{latexsym}
\usepackage{amsmath,amssymb, amsthm, stmaryrd, bm, bbm}
\usepackage[all,2cell]{xy}
\usepackage{mathrsfs}
\usepackage{color}

\newtheorem{theorem}[subsection]{Theorem}
\newtheorem*{theorem*}{Theorem}

\newtheorem{lemma}[subsection]{Lemma}
\newtheorem{proposition}[subsection]{Proposition}

\newtheorem*{conjecture*}{Conjecture}

\newtheorem*{question*}{Question}
\theoremstyle{remark}
\newtheorem{remark}[subsection]{Remark}
\newtheorem{example}[subsection]{Example}
\theoremstyle{definition}
\newtheorem{definition}[subsection]{Definition}
\newtheorem*{notation*}{Notation}

\newcommand{\opname}[1]{\operatorname{\mathsf{#1}}}

\newcommand{\ca}{{\mathcal A}}

\newcommand{\cc}{{\mathcal C}}

\numberwithin{equation}{section}

\begin{document}

\title[]{Reduction of wide subcategories and recollements}

{\author{Yingying Zhang}


\vspace{0.2cm}

\address{Department of Mathematics, Huzhou University, Huzhou 313000, Zhejiang Province, P.R.China}

\email{yyzhang@zjhu.edu.cn}

\begin{abstract}
In this paper, we prove a reduction result on wide subcategories of abelian categories which is similar to Calabi-Yau reduction, silting reduction and $\tau$-tilting reduction. More precisely, if an abelian category $\mathcal{A}$ admits a recollement relative to abelian categories $\mathcal{A}'$ and $\mathcal{A}''$, diagrammatically expressed by
$$\xymatrix@!C=2pc{
\mathcal{A'} \ar@{>->}[rr]|{i_{*}} && \mathcal{A} \ar@<-4.0mm>@{->>}[ll]_{i^{*}} \ar@{->>}[rr]|{j^{*}} \ar@{->>}@<4.0mm>[ll]^{i^{!}}&& \mathcal{A''} \ar@{>->}@<-4.0mm>[ll]_{j_{!}} \ar@{>->}@<4.0mm>[ll]^{j_{*}}
},$$
then the assignment $\cc\mapsto j^*(\cc)$ defines a bijection between wide subcategories in $\mathcal{A}$ containing $i_{*}(\mathcal{A}')$ and wide subcategories in $\mathcal{A}''$. Moreover, a wide subcategory $\mathcal{C}$ of $\mathcal{A}$ containing $i_{*}(\mathcal{A}')$ admits a new recollement relative to $\mathcal{A}'$ and $j^{*}(\mathcal{C})$ which is induced from the original recollement.
\end{abstract}

\maketitle

\section{Introduction}\label{s:introduction}
\medskip

Wide subcategories of an abelian category are analogous to thick subcategories of a triangulated category.
This notion was introduced by Hovey in [Hov] and provides a significant interface between representation theory and combinatorics [IT].  Wide subcategories are closely related to other important notions in representation theory, such as semibricks [As], torsion classes [AIR, DIRRT, MS], silting complexes [AIR], semistable subcategories [Y] and support $\tau$-tilting modules [AIR]. Wide subcategories of $d$-abelian categories was introduced and studied in [HJV].

Recollements of abelian categories were introduced by Beilinson, Bernstein and Deligne in [BBD] for the construction of the category of perverse sheaves on a singular space, arising from recollements of triangulated categories. The notion of recollements is very useful in gluing and in reduction, see for example [AKL, BBD, C, FP, LVY, P].
The main result of this paper is the following reduction result for wide subcategories. This is similar to Calabi-Yau reduction in [IY2], silting reduction in [AI] and [IY1], and reduction of support $\tau$-tilting modules in [EJR] and [J].

\begin{theorem}{\rm(Theorems 3.4 and 3.5)}
\label{thm:main-theorem}
Let ($\mathcal{A}', \mathcal{A}, \mathcal{A}'', i^{*}, i_{*}, i^{!}, j_{!}, j^{*}, j_{*}$) be a recollement (see Definition 2.4 for details). Then we have
\begin{enumerate}
\item There is a bijection between wide subcategories in $\mathcal{A}$ containing $i_{*}(\mathcal{A}')$ and wide subcategories in $\mathcal{A}''$.
\item Let $\mathcal{C}$ be a subcategory of $\mathcal{A}$. If $i_{*}i^{*}(\mathcal{C})\subset\mathcal{C}$(resp. $  i_{*}i^{!}(\mathcal{C})\subset\mathcal{C})$, then $\mathcal{C}\subset\mathcal{A}$ is wide implies that $i^{*}(\mathcal{C})$ (resp. $i^{!}(\mathcal{C})$)$\subset\mathcal{A}'$ is wide.
\end{enumerate}
\end{theorem}
Note that $i_{*}i^{*}(\mathcal{C})\subset\mathcal{C}$ (resp. $ i_{*}i^{!}(\mathcal{C})\subset\mathcal{C})$ is weaker than $i_{*}(\mathcal{A}')\subset\mathcal{C}$.
It is a general interesting problem to construct a new recollement from a known one, see for example [CL, LL, LW] for some study on this problem. As an application of Theorem~\ref{thm:main-theorem}, we can construct a new recollement from the given one as follows:
\begin{theorem}{\rm(Theorem 3.8)}
\label{thm:main-theorem-2}
Let ($\mathcal{A}', \mathcal{A}, \mathcal{A}'', i^{*}, i_{*}, i^{!}, j_{!}, j^{*}, j_{*}$) be a recollement of abelian categories and $\mathcal{C}$ be a wide subcategory of $\mathcal{A}$. If $i_{*}(\mathcal{A}')\subset\mathcal{C}$, then the given recollement restricts to a new recollement of wide subcategories ($\mathcal{A}'$, $\mathcal{C}, j^{*}(\mathcal{C}), \overline{i^{*}}, \overline{i_{*}}, \overline{i^{!}}, \overline{j_{!}}, \overline{j^{*}}, \overline{j_{*}}$).
\end{theorem}

\vspace{0.2cm}

\section{Preliminaries}
\medskip

In this section, we collect some basic material that will be used later. We begin this section by introducing the definition of wide subcategories.
\begin{definition}{[\rm Hov, Definition 1.1]}
A full subcategory $\mathcal{C}$ of an abelian category $\mathcal{A}$ is called $wide$ if it is abelian and closed under extensions.
\end{definition}
Note that when we say a full subcategory $\mathcal{C}$ is abelian, we mean that $\mathcal{C}$ is closed under kernels and cokernels, more precisely, if $f: M\rightarrow N$ is a morphism of $\mathcal{C}$, then both the kernel and cokernel of $f$ are also in $\mathcal{C}$. We denote by $\opname{wide}\mathcal{A}$ the set of wide subcategories of $\mathcal{A}$ and $\opname{wide}_{\mathcal{C}}\mathcal{A}$ the set of wide subcategories of $\mathcal{A}$ containing a given subcategory $\mathcal{C}$.
\begin{example}{\rm[Hov]}
Let $R$ be a ring.
\begin{enumerate}
\item The empty subcategory, the zero subcategory and the entire category of $R$-modules are all wide subcategories.
\item The category of all $R$-modules of cardinality $\leq\kappa$ for some infinite cardinal $\kappa$ is a wide subcategory.
\item The category of finite-dimensional rational vector spaces, as a subcategory of the category of abelian groups, is a wide subcategory.
\end{enumerate}
\end{example}
For the convenience of the reader, we recall the definition of adjoint pairs [ARS] before the definition of recollements.
\begin{definition}
Let $\mathcal{A}$ and $\mathcal{B}$ be categories. An $adjunction$ from $\mathcal{A}$ to $\mathcal{B}$ is a triple $\langle F, G, \phi\rangle: \mathcal{A}\rightarrow\mathcal{B}$, where $F$ is a functor from $\mathcal{A}$ to $\mathcal{B}$, $G$ is a functor from $\mathcal{B}$ to $\mathcal{A}$, and $\phi$ is a functor which assigns to each pair of objects $A\in\mathcal{A}, B\in\mathcal{B}$, a bijection $\phi=\phi_{A, B}: {\rm Hom}_{\mathcal{B}}(FA, B)\cong{\rm Hom}_{\mathcal{A}}(A, GB)$ which is functorial in $A$ and $B$. The functor $F$ is called a $left$ $adjoint$ of $G$, $G$ is called a $right$ $adjoint$ of $F$, and the pair $(F, G)$ is called an $adjoint$ $pair$ from $\mathcal{A}$ to $\mathcal{B}$.
\end{definition}
Now we recall the definition of recollements of abelian categories, see for instance [BBD, FP, K].

\begin{definition} Let $\mathcal{A}, \mathcal{A'}, \mathcal{A''}$ be abelian categories. Then a recollement of $\mathcal{A}$ relative to $\mathcal{A'}$ and $\mathcal{A''}$ is given by the following diagram
$$\xymatrix@!C=2pc{
\mathcal{A'} \ar@{>->}[rr]|{i_{*}} && \mathcal{A} \ar@<-4.0mm>@{->>}[ll]_{i^{*}} \ar@{->>}[rr]|{j^{*}} \ar@{->>}@<4.0mm>[ll]^{i^{!}}&& \mathcal{A''} \ar@{>->}@<-4.0mm>[ll]_{j_{!}} \ar@{>->}@<4.0mm>[ll]^{j_{*}}
}$$
such that
\begin{enumerate}
\item ($i^{*}, i_{*}$), ($i_{*}, i^{!}$), ($j_{!}, j^{*}$) and ($j^{*}, j_{*}$) are adjoint pairs;
\item $i_{*}, j_{!}$ and $j_{*}$ are fully faithful functors;
\item ${\rm Im}i_{*}={\rm Ker}j^{*}$.
\end{enumerate}
Throughout this paper, we denote by ($\mathcal{A}', \mathcal{A}, \mathcal{A}'', i^{*}, i_{*}, i^{!}, j_{!}, j^{*}, j_{*}$) a recollement of $\mathcal{A}$ relative to $\mathcal{A'}$ and $\mathcal{A''}$ as above.
\end{definition}
\begin{remark}
(1) From Definition 2.4(1), it follows that $i_{*}$ and $j^{*}$ are both right adjoint functors and left adjoint functors, therefore they are exact functors of abelian categories.

(2) By Definition 2.4(2), one can assume that $i^{*}i_{*}\cong id, i^{!}i_{*}\cong id, j^{*}j_{!}\cong id$ and $j^{*}j_{*}\cong id$ in the definition of recollements.

(3) By [PV, Proposition 2.8], for any $A\in\mathcal{A}$, there exist $M'$ and $N'$ in $\mathcal{A}'$ such that the units and counits of the adjunctions induce the following exact sequences
$$ 0\rightarrow i_{*}(M')\rightarrow j_{!}j^{*}(A)\rightarrow A\rightarrow i_{*}i^{*}(A)\rightarrow 0,$$
$$0\rightarrow i_{*}i^{!}(A)\rightarrow A\rightarrow j_{*}j^{*}(A)\rightarrow i_{*}(N')\rightarrow 0.
$$
\end{remark}

\vspace{0.2cm}

\section{Recollements of wide subcategories}
\medskip
In this section we prove Theorems~\ref{thm:main-theorem} and~\ref{thm:main-theorem-2}.

Throughout this section, let ($\mathcal{A}', \mathcal{A}, \mathcal{A}'', i^{*}, i_{*}, i^{!}, j_{!}, j^{*}, j_{*}$) be a recollement of abelian categories.
We first establish the following important result.

\begin{lemma}
Let $\mathcal{C}\in \opname{wide}_{i_{*}(\mathcal{A}')}\mathcal{A}$. Then we have
\begin{enumerate}
\item $i_{*}i^{*}(\mathcal{C})\subset \mathcal{C}$, $i_{*}i^{!}(\mathcal{C})\subset \mathcal{C}$, $j_{*}j^{*}(\mathcal{C})\subset \mathcal{C}$ and $j_{!}j^{*}(\mathcal{C})\subset \mathcal{C}$.
\item If $j^{*}(C)\in j^{*}(\mathcal{C})$ for some $C\in \mathcal{A}$, then $C\in \mathcal{C}$.
\end{enumerate}
\end{lemma}

\begin{proof}
(1) Since $i_{*}(\mathcal{A}')\subset \mathcal{C}$, the inclusions $i_{*}i^{*}(\mathcal{C})\subset \mathcal{C}$ and $i_{*}i^{!}(\mathcal{C})\subset \mathcal{C}$ are obvious.
For any $C\in \mathcal{C}\subset \mathcal{A}$, by Remark 2.5(3) there exist $M',N'\in\mathcal{A}$ and two exact sequences
$$ 0\rightarrow i_{*}(M')\rightarrow j_{!}j^{*}(C)\rightarrow C\rightarrow i_{*}i^{*}(C)\rightarrow 0,$$
$$0\rightarrow i_{*}i^{!}(C)\rightarrow C\rightarrow j_{*}j^{*}(C)\rightarrow i_{*}(N')\rightarrow 0.
$$ Since $\mathcal{C}$ is wide and $i_{*}(\mathcal{A}')\subset\mathcal{C}$, it follows that $j_{*}j^{*}(C)\in\mathcal{C}$ and $j_{!}j^{*}(C)\in\mathcal{C}$. Therefore $j_{*}j^{*}(\mathcal{C})\subset \mathcal{C}$ and $j_{!}j^{*}(\mathcal{C})\subset \mathcal{C}$.

(2) By Remark 2.5(3) there exists an exact sequence
$$0\rightarrow i_{*}(C')\stackrel{f}{\rightarrow} j_{!}j^{*}(C)\rightarrow C\rightarrow i_{*}i^{*}(C)\rightarrow 0,$$
where $C'\in \mathcal{A}'$. Since $j^*(C)\in j^*(\cc)$, by (1) we have $j_!j^*(C)\in j_!j^*(\cc)\subset \cc$. Since $i_*(\ca')\subset\cc$, both $i_*(C')$ and $i_*i^*(C)$ are in $\cc$. Therefore $f$ is a morphism in $\cc$, and thus its cokernel $\mathrm{Coker}f$ belongs to $\cc$ because $\cc$ is closed under cokernels. Now $C$, being an extension of $i_*i^*(C)$ by $\mathrm{Coker}f$, belongs to $\cc$, because $\cc$ is closed under extensions.
\end{proof}

Now we are in the position to prove Theorem 1.1.

\begin{proposition}
If $\mathcal{C}\in\opname{wide}_{i_{*}(\mathcal{A}')}\mathcal{A}$, then $j^{*}(\mathcal{C})\in\opname{wide}\mathcal{A}''$.
\end{proposition}
\begin{proof}
Assume that $\mathcal{C}$ is wide and $i_{*}(\mathcal{A}')\subset \mathcal{C}$.

Step 1: $j^{*}(\mathcal{C})$ is closed under kernels and cokernels. If $f'': M''\rightarrow N''$ is a morphism of $j^{*}(\mathcal{C})$, we have to prove that ${\rm Ker}f'', {\rm Coker}f''\in j^{*}(\mathcal{C})$. Applying the functor $j_{*}$ to $f''$, we have the following exact sequence:
\begin{equation}
0\longrightarrow {\rm Ker}j_{*}(f'')\longrightarrow j_{*}(M'')\buildrel {j_{*}(f'')} \over\longrightarrow j_{*}(N'')\longrightarrow{\rm Coker}j_{*}(f'')\longrightarrow 0.
\end{equation}
By Lemma 3.1(1), $j_{*}j^{*}(\mathcal{C})\subset\mathcal{C}$, it follows that $j_{*}(f'')$ is a morphism of $\mathcal{C}$. Since $\mathcal{C}$ is wide, ${\rm Ker}j_{*}(f''), {\rm Coker}j_{*}(f'')\in\mathcal{C}$.
Then applying $j^{*}$ to (3.1), by Remark 2.5(1) and (2) we have the following commutative diagram:
$$\xymatrix{
0\ar[r] & j^{*}({\rm Ker}j_{*}(f''))\ar[r]\ar@{.>}[d] & j^{*}j_{*}(M'')\ar[r]^{j^{*}j_{*}(f'')}\ar[d]^{\cong} & j^{*}j_{*}(N'')\ar[d]^{\cong}\ar[r]\ar[d]^{\cong} & j^{*}({\rm Coker}j_{*}(f''))\ar[r]\ar@{.>}[d] & 0\\
0\ar[r] & {\rm Ker}f''\ar[r] & M''\ar[r]^{f''} & N''\ar[r] & {\rm Coker}f''\ar[r] & 0
.}$$
Thus ${\rm Ker}f''\cong j^{*}({\rm Ker}j_{*}(f''))\in j^{*}(\mathcal{C})$ and ${\rm Coker}f''\cong j^{*}({\rm Coker}j_{*}(f''))\in j^{*}(\mathcal{C})$.

Step 2: $j^{*}(\mathcal{C})$ is closed under extensions. Let $0\rightarrow M'' \buildrel {f''} \over\rightarrow N''\rightarrow P''\rightarrow 0$ be an exact sequence in $\mathcal{A''}$ where $M'', P''\in j^{*}(\mathcal{C})$, we have to prove that $N''\in j^{*}(\mathcal{C})$. Since $j_{*}$ is left exact, applying it to $f''$, we have the following exact sequence:
\begin{equation}
0\longrightarrow  j_{*}(M'')\buildrel {j_{*}(f'')} \over\longrightarrow j_{*}(N'')\longrightarrow {\rm Coker}j_{*}(f'')\longrightarrow 0.
\end{equation}
By Lemma 3.1(1), $j_{*}j^{*}(\mathcal{C})\subset\mathcal{C}$. It follows that $j_{*}(M'')\in\mathcal{C}$. Then applying $j^{*}$ to (3.2), by Remark 2.5(1) and (2) we have the following commutative diagram:
$$\xymatrix{
0\ar[r] & j^{*}j_{*}(M'')\ar[r]^{j^{*}j_{*}(f'')}\ar[d]^{\cong} & j^{*}j_{*}(N'')\ar[r]\ar[d]^{\cong} & j^{*}({\rm Coker}j_{*}(f''))\ar[r]\ar@{.>}[d] & 0\\
0\ar[r] & M''\ar[r]^{f''} & N''\ar[r] & P''\ar[r] & 0
.}$$
By uniqueness of the cokernel, we know that $P''\cong j^{*}({\rm Coker}j_{*}(f''))\in j^{*}(\mathcal{C})$. By Lemma 3.1(2), ${\rm Coker}j_{*}(f'')\in \mathcal{C}$. Since $\mathcal{C}$ is wide, we know that $j_{*}(N'')\in \mathcal{C}$, which is the middle term of the exact sequence (3.2). So $N''\cong j^{*}j_{*}(N'')\in j^{*}(\mathcal{C})$.

We have finished to prove that $j^{*}(\mathcal{C})$ is wide.
\end{proof}
Conversely, starting from a wide subcategory of $\mathcal{A''}$ we can also find a corresponding wide subcategory of $\mathcal{A}$.
\begin{proposition}
If $\mathcal{W}\in\opname{wide}\mathcal{A''}$, then $\mathcal{C}=\{M\in\mathcal{A}|j^{*}(M)\in\mathcal{W}\}\in\opname{wide}_{i_{*}(\mathcal{A}')}\mathcal{A}$.
\end{proposition}
\begin{proof}
If $f$ is a morphism of $\mathcal{C}$, then $j^{*}(f)$ is a morphism of $\mathcal{W}$ by the definition of $\mathcal{C}$. Since $\mathcal{W}$ is wide, by Remark 2.5(1) we have $j^{*}({\rm Ker}f)\cong {\rm Ker}j^{*}(f)\in\mathcal{W}$ and $j^{*}({\rm Coker}f)\cong{\rm Coker}j^{*}(f)\in\mathcal{W}$. Thus ${\rm Ker}f$ and ${\rm Coker}f$ are both in $\mathcal{C}$.

Let $0\rightarrow M \rightarrow N\rightarrow P\rightarrow 0$ be an exact sequence in $\mathcal{A}$ where $M, P\in\mathcal{C}$, we only have to prove that $N\in\mathcal{C}$. Since the functor $j^{*}$ is exact, we have the following exact sequence:
$$
0\longrightarrow j^{*}(M)\longrightarrow j^{*}(N)\longrightarrow j^{*}(P)\longrightarrow 0.
$$
Note that $j^{*}(M)\in\mathcal{W}$ and $j^{*}(P)\in\mathcal{W}$, it follows that $j^{*}(N)\in\mathcal{W}$ since $\mathcal{W}$ is wide. Thus $N\in\mathcal{C}$. We conclude that $\mathcal{C}$ is a wide subcategory of $\mathcal{A}$. It is obvious that $i_{*}(\mathcal{A}')\subset\mathcal{C}$ since $j^{*}i_{*}=0$.
\end{proof}
In conclusion, we have the following theorem.
\begin{theorem}
There is a bijection
$$\opname{wide}_{i_{*}(\mathcal{A}')}\mathcal{A}\leftrightarrow \opname{wide}\mathcal{A}''$$ given by $\opname{wide}_{i_{*}(\mathcal{A}')}\mathcal{A}\ni\mathcal{C}\mapsto j^{*}(\mathcal{C})\in\opname{wide}\mathcal{A}''$ and $\opname{wide}\mathcal{A}''\ni\mathcal{W}\mapsto \mathcal{C}=\{M\in\mathcal{A}|j^{*}(M)\in\mathcal{W}\}\in\opname{wide}_{i_{*}(\mathcal{A}')}\mathcal{A}$.
\end{theorem}
\begin{proof}
By Proposition 3.2 and Proposition 3.3, we only have to prove that the maps above are inverse to each other.
Let $\mathcal{W}\in\opname{wide}\mathcal{A}''$ and $\mathcal{C}=\{M\in\mathcal{A}|j^{*}(M)\in\mathcal{W}\}$. We claim that $j^{*}(\mathcal{C})=\mathcal{W}$. Clearly $j^{*}(\mathcal{C})\subset\mathcal{W}$. For any $N\in\mathcal{W}$, $N\cong j^{*}j_{*}(N)\in\mathcal{W}$, then $j_{*}(N)\in\mathcal{C}$ and therefore $N\in j^{*}(\mathcal{C})$.

Conversely, let $\mathcal{C}\in\opname{wide}_{i_{*}(\mathcal{A}')}\mathcal{A}$ and $\mathcal{C}'=\{M\in\mathcal{A}|j^{*}(M)\in j^{*}(\mathcal{C})\}\in\opname{wide}_{i_{*}(\mathcal{A}')}\mathcal{A}$. We claim that $\mathcal{C}=\mathcal{C}'$. Clearly $\mathcal{C}\subset\mathcal{C}'$. For any $M\in\mathcal{C}'\subset\mathcal{A}$, by Lemma 3.1(2) we have $M\in\mathcal{C}$. We have finished to prove that $\mathcal{C}=\mathcal{C}'$.
\end{proof}
\begin{theorem}
If $\mathcal{C}\subset\mathcal{A}$ is wide and satisfies $i_{*}i^{*}(\mathcal{C})\subset\mathcal{C}$(resp. $i_{*}i^{!}(\mathcal{C})\subset\mathcal{C})$), then $i^{*}(\mathcal{C})$ (resp. $i^{!}(\mathcal{C})$)$\subset\mathcal{A}'$ is wide.
\end{theorem}
\begin{proof}
Assume that $\mathcal{C}\subset\mathcal{A}$ is wide and satisfies $i_{*}i^{*}(\mathcal{C})\subset\mathcal{C}$. We prove that $i^{*}(\mathcal{C})\subset\mathcal{A}'$ is wide. The proof that $i^{!}(\mathcal{C})\subset\mathcal{A}'$ is wide when $i_{*}i^{!}(\mathcal{C})\subset\mathcal{C}$ is similar.

Step 1: $i^{*}(\mathcal{C})$ is closed under kernels and cokernels. If $f': M'\rightarrow N'$ is a morphism of $i^{*}(\mathcal{C})$, we have to prove that ${\rm Ker}f',~{\rm Coker}f'\in i^{*}(\mathcal{C})$. Applying the functor $i_{*}$ to $f'$, by Remark 2.5(1) we have the following exact sequence:
$$
0\longrightarrow i_{*}({\rm Ker}f')\longrightarrow i_{*}(M')\buildrel {i_{*}(f')} \over\longrightarrow i_{*}(N')\longrightarrow i_{*}({\rm Coker}f')\longrightarrow 0.
$$
Note that $i_{*}i^{*}(\mathcal{C})\subset\mathcal{C}$, it follows that $i_{*}(f')$ is a morphism of $\mathcal{C}$. Since $\mathcal{C}$ is wide, it follows that $i_{*}({\rm Ker}f')$ and $i_{*}({\rm Coker}f')$ belong to $\mathcal{C}$. Thus ${\rm Ker}f'\cong i^{*}i_{*}({\rm Ker}f')\in i^{*}(\mathcal{C})$ and ${\rm Coker}f'\cong i^{*}i_{*}({\rm Coker}f')\in i^{*}(\mathcal{C})$.

Step 2: $i^{*}(\mathcal{C})$ is closed under extensions. Let $0\rightarrow M' \rightarrow N'\rightarrow P'\rightarrow 0$ be an exact sequence in $\mathcal{A}'$ where $M', P'\in i^{*}(\mathcal{C})$, we have to prove that $N'\in i^{*}(\mathcal{C})$. Applying the functor $i_{*}$ to the above exact sequence, by Remark 2.5(1) we have the following exact sequence:
$$
0\longrightarrow i_{*}(M')\longrightarrow i_{*}(N')\longrightarrow i_{*}(P')\longrightarrow 0.
$$
Note that $i_{*}i^{*}(\mathcal{C})\subset\mathcal{C}$, it follows that $i_{*}(M'), i_{*}(P')\in \mathcal{C}$. Since $\mathcal{C}$ is wide, $i_{*}(N')\in \mathcal{C}$. Thus $N'\cong i^{*}i_{*}(N')\in i^{*}(\mathcal{C})$.

We have finished to prove that $i^{*}({\mathcal{C}})$ is wide.
\end{proof}

The next result states that an adjoint pair of two categories can be restricted to an adjoint pair of subcategories under an additional condition.

\begin{proposition}
Let $(F, G)$ be an adjoint pair from $\mathcal{A}$ to $\mathcal{B}$ as in Definition 2.3. If $\mathcal{C}$ is a subcategory of $\mathcal{A}$ satisfying $GF(\mathcal{C})\subset\mathcal{C}$, then the restricted functors $(\overline{F}, \overline{G})(=$$(F|_{\mathcal{C}}, G|_{F(\mathcal{C})}))$ is an adjoint pair from $\mathcal{C}$ to $F(\mathcal{C})$.

Dually, if $\mathcal{D}$ is a subcategory of $\mathcal{B}$ satisfying $FG(\mathcal{D})\subset\mathcal{D}$, then the restricted functors $(\overline{F}, \overline{G})(=$$(F|_{G(\mathcal{D})}, G|_{\mathcal{D}}))$ is an adjoint pair from $G(\mathcal{D})$ to $\mathcal{D}$.
\end{proposition}
\begin{proof}
For any $C, D\in\mathcal{C}$, since $(F, G)$ is an adjoint pair from $\mathcal{A}$ to $\mathcal{B}$ and $GF(\mathcal{C})\subset\mathcal{C}$, we have the following bijection:
$${\rm Hom}_{F(\mathcal{C})}(\overline{F}C, FD)={\rm Hom}_{\mathcal{B}}(FC, FD)\cong{\rm Hom}_{\mathcal{A}}(C, GFD)={\rm Hom}_{\mathcal{C}}(C, \overline{G}FD),$$ which is functorial in $C$ and $D$. Therefore $(\overline{F}, \overline{G})$ is an adjoint pair  from to $\mathcal{C}$ to $F(\mathcal{C})$.
\end{proof}

\begin{lemma}
If $i_{*}(\mathcal{A}')\subset \mathcal{C}$, then $i^{*}(\mathcal{C})=i^{!}(\mathcal{C})=\mathcal{A}'$.
\end{lemma}
\begin{proof}
In fact, if $i_{*}(\mathcal{A}')\subset \mathcal{C}$, then $\mathcal{A}'\cong i^{*}i_{*}(\mathcal{A}')\subset i^{*}(\mathcal{C})$ and $\mathcal{A}'\cong i^{!}i_{*}(\mathcal{A}')\subset i^{!}(\mathcal{C})$. Thus $i^{*}(\mathcal{C})=i^{!}(\mathcal{C})=\mathcal{A}'$.
\end{proof}

Now we give the proof of Theorem 1.2.

\begin{theorem}
If $\mathcal{C}\subset\mathcal{A}$ is wide and satisfies $i_{*}(\mathcal{A}')\subset\mathcal{C}$, then we obtain a recollement of wide subcategories as follows:
$$\xymatrix@!C=2pc{
\mathcal{A}' \ar@{>->}[rr]|{\overline{i_{*}}} && \mathcal{C} \ar@<-4.0mm>@{->>}[ll]_{\overline{i^{*}}} \ar@{->>}[rr]|{\overline{j^{*}}} \ar@{->>}@<4.0mm>[ll]^{\overline{i^{!}}}&& j^{*}(\mathcal{C}) \ar@{>->}@<-4.0mm>[ll]_{\overline{j_{!}}} \ar@{>->}@<4.0mm>[ll]^{\overline{j_{*}}}
}.$$
\end{theorem}

\begin{proof} Of course, $\ca'$ is wide in $\ca'$.
By Proposition 3.2, $j^{*}(\mathcal{C})$ is wide  in $\ca''$ since $\mathcal{C}$ is wide satisfying $i_{*}(\mathcal{A}')\subset\mathcal{C}$.

Step 1: By Lemma 3.1(1), Propsition 3.6 and Lemma 3.7, we know that ($\overline{i^{*}}, \overline{i_{*}}$), ($\overline{i_{*}}, \overline{i^{!}}$), ($\overline{j_{!}}, \overline{j^{*}}$) and ($\overline{j^{*}}, \overline{j_{*}}$) are adjoint pairs.

Step 2: By Definition 2.4(2), we know that $i_{*}, j_{!}$ and $j_{*}$ are fully faithful. So the restriction functors $\overline{i_{*}}, \overline{j_{!}}$ and $\overline{j_{*}}$ are also fully faithful.

Step 3: ${\rm Im}\overline{i_{*}}= {\rm Ker}\overline{j^{*}}$ is obvious.

By the definition of recollements, we have done.
\end{proof}

Here we give an easy example to illustrate the result.
\begin{example}
Let $\Lambda$ be the path algebra over a field $K$ of the quiver $1\leftarrow 2\rightarrow 3$, of type $A_{3}$. If $e$ is the idempotent $e_{2}+e_{3}$, then as a right $\Lambda$-module $\Lambda/\Lambda e\Lambda$ is isomorphic to the simple module $S_{1}$. In this case, there is a recollement as follows:
$$\xymatrix@!C=2pc{
{\rm mod\,}(\Lambda/\Lambda e\Lambda)\ar@{>->}[rr]|{i_{*}} && {\rm mod\,}\Lambda \ar@<-4.0mm>@{->>}[ll]_{i^{*}} \ar@{->>}[rr]|{j^{*}} \ar@{->>}@<4.0mm>[ll]^{i^{!}}&& {\rm mod\,}(e\Lambda e)\ar@{>->}@<-4.0mm>[ll]_{j_{!}} \ar@{>->}@<4.0mm>[ll]^{j_{*}}
,}$$
where $i_{*}=-\otimes_{\Lambda/\Lambda e\Lambda}\Lambda/\Lambda e \Lambda$, $i^{*}=-\otimes_{\Lambda}\Lambda/\Lambda e \Lambda$, $i^{!}={\rm Hom}_{\Lambda}(\Lambda/\Lambda e\Lambda, -)$, $j_{!}=-\otimes_{e\Lambda e}e \Lambda$, $j_{*}={\rm Hom}_{e\Lambda e}(\Lambda e, -)$, $j^{*}=(-)e$.

The Auslander-Reiten quiver of ${\rm mod\,}\Lambda$ is the following:
$$
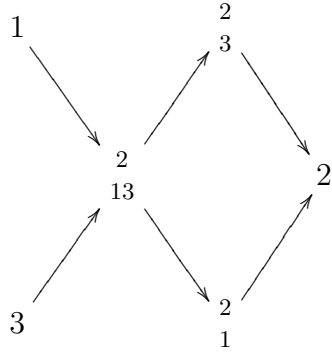
\xymatrix@R=1pc{
&1\ar[dr]&&\scriptsize{\begin{matrix} 2\\3\\ \end{matrix}}\ar[dr]&&\\
&& \scriptsize{\begin{matrix} 2\\13\\ \end{matrix}}\ar[ur]\ar[dr]&&2&\\
&3\ar[ur]&& \scriptsize{\begin{matrix} 2\\1\\ \end{matrix}}\ar[ur]&&\\
}$$

\begin{table}[htbp]
  \caption{\label{tab:test}}
 \begin{tabular}{lcl}
  \toprule
 {\rm mod\,}($\Lambda/\Lambda e\Lambda$)  & $\mathcal{C}$ \quad\quad\quad & $j^{*}(\mathcal{C})$ \\
  \midrule
  \begin{xy}
     (0,0)="0"*{{\color{red} \bullet}}, +(4,0)*{\circ},
     "0"+(2,-3.5)*{\circ}, +(4,0)*{\circ},
     "0"+(0,-7)*{\circ}, +(4,0)*{\circ}, +(0,-2.3)*{}
  \end{xy} &
  \begin{xy}
     (0,0)="0"*{{\color{red} \bullet}}, +(4,0)*{\color{red} \bullet},
     "0"+(2,-3.5)*{\color{red} \bullet}, +(4,0)*{\color{red} \bullet},
     "0"+(0,-7)*{\color{red} \bullet}, +(4,0)*{\color{red} \bullet}, +(0,-2.3)*{}
  \end{xy} \quad\quad\quad &
  \begin{xy}
     (0,0)="0",
     "0"+(2,-3.5)*{\color{red} \bullet},
     "0"+(0,-7)*{\color{red} \bullet}, +(4,0)*{\color{red} \bullet}, +(0,-2.3)*{}
  \end{xy} \\
  \begin{xy}
     (0,0)="0"*{{\color{red} \bullet}}, +(4,0)*{\circ},
     "0"+(2,-3.5)*{\circ}, +(4,0)*{\circ},
     "0"+(0,-7)*{\circ}, +(4,0)*{\circ}, +(0,-2.3)*{}
  \end{xy} &
  \begin{xy}
     (0,0)="0"*{{\color{red} \bullet}}, +(4,0)*{\color{red} \bullet},
     "0"+(2,-3.5)*{\color{red} \bullet}, +(4,0)*{\circ},
     "0"+(0,-7)*{\circ}, +(4,0)*{\circ}, +(0,-2.3)*{}
  \end{xy}\quad\quad\quad &
  \begin{xy}
     (0,0)="0",
     "0"+(2,-3.5)*{\color{red} \bullet},
     "0"+(0,-7)*{\circ}, +(4,0)*{\circ}, +(0,-2.3)*{}
  \end{xy} \\
  \begin{xy}
     (0,0)="0"*{{\color{red} \bullet}}, +(4,0)*{\circ},
     "0"+(2,-3.5)*{\circ}, +(4,0)*{\circ},
     "0"+(0,-7)*{\circ}, +(4,0)*{\circ}, +(0,-2.3)*{}
  \end{xy} &
  \begin{xy}
     (0,0)="0"*{{\color{red} \bullet}}, +(4,0)*{\circ},
     "0"+(2,-3.5)*{\circ}, +(4,0)*{\circ},
     "0"+(0,-7)*{\color{red} \bullet}, +(4,0)*{\circ}, +(0,-2.3)*{}
  \end{xy}\quad\quad\quad &
  \begin{xy}
     (0,0)="0",
     "0"+(2,-3.5)*{\circ},
     "0"+(0,-7)*{\color{red} \bullet}, +(4,0)*{\circ}, +(0,-2.3)*{}
  \end{xy} \\
  \begin{xy}
     (0,0)="0"*{{\color{red} \bullet}}, +(4,0)*{\circ},
     "0"+(2,-3.5)*{\circ}, +(4,0)*{\circ},
     "0"+(0,-7)*{\circ}, +(4,0)*{\circ}, +(0,-2.3)*{}
  \end{xy} &
  \begin{xy}
     (0,0)="0"*{{\color{red} \bullet}}, +(4,0)*{\circ},
     "0"+(2,-3.5)*{\circ}, +(4,0)*{\color{red} \bullet},
     "0"+(0,-7)*{\circ}, +(4,0)*{\color{red} \bullet}, +(0,-2.3)*{}
  \end{xy}\quad\quad\quad &
  \begin{xy}
     (0,0)="0",
     "0"+(2,-3.5)*{\circ},
     "0"+(0,-7)*{\circ}, +(4,0)*{\color{red} \bullet}, +(0,-2.3)*{}
  \end{xy} \\
  \begin{xy}
     (0,0)="0"*{{\color{red} \bullet}}, +(4,0)*{\circ},
     "0"+(2,-3.5)*{\circ}, +(4,0)*{\circ},
     "0"+(0,-7)*{\circ}, +(4,0)*{\circ}, +(0,-2.3)*{}
  \end{xy} &
  \begin{xy}
     (0,0)="0"*{{\color{red} \bullet}}, +(4,0)*{\circ},
     "0"+(2,-3.5)*{\circ}, +(4,0)*{\circ},
     "0"+(0,-7)*{\circ}, +(4,0)*{\circ}, +(0,-2.3)*{}
  \end{xy}\quad\quad\quad &
  \begin{xy}
     (0,0)="0",
     "0"+(2,-3.5)*{\circ},
     "0"+(0,-7)*{\circ}, +(4,0)*{\circ}, +(0,-2.3)*{}
  \end{xy} \\
  \bottomrule
 \end{tabular}
\end{table}

In Table 1, we give a complete list of wide subcategories $\mathcal{C}$ of $\mathcal{A}$ which satisfies $i_{*}({\rm mod\,}(\Lambda/\Lambda e\Lambda))\subset\mathcal{C}$, also each of the corresponding wide subcategories $j^{*}(\mathcal{C})$ in ${\rm mod\,}(e\Lambda e)$ is listed in the same row. The subcategories of module categories are indicated by specifying a subset of the indecomposable modules in the Auslander-Reiten quiver.

\end{example}

\vspace{0.2cm}
\bigskip
{\bf Acknowledgement.} The author would like to thank Dong Yang for the useful comments. This work was supported by Huzhou University's scientific research project(Grant No. 2021XJKJ09).

\end{document}